\documentclass[a4paper,10pt]{article}
\usepackage{amsmath,amssymb,amsthm,graphicx, url}
\usepackage[all]{xypic}
\usepackage[symbol]{footmisc}
\DeclareMathOperator{\A}{\mathbb{A}}
\DeclareMathOperator{\Br}{Br}
\DeclareMathOperator{\cd}{cd}

\DeclareMathOperator{\Div}{Div}
\DeclareMathOperator{\End}{End}

\DeclareMathOperator{\F}{\mathbb{F}}
\DeclareMathOperator{\Gal}{Gal}
\DeclareMathOperator{\G}{\mathbb{G}}
\DeclareMathOperator{\Hom}{Hom}
\DeclareMathOperator{\Ker}{Ker}
\DeclareMathOperator{\cO}{\mathcal{O}}
\DeclareMathOperator{\cP}{\mathcal{P}}
\DeclareMathOperator{\Pic}{Pic}
\DeclareMathOperator{\Proj}{Proj}
\DeclareMathOperator{\Q}{\mathbb{Q}}
\DeclareMathOperator{\Spec}{Spec}
\DeclareMathOperator{\res}{res}
\DeclareMathOperator{\ur}{ur}
\DeclareMathOperator{\cV}{\mathcal{V}}

\DeclareMathOperator{\Z}{\mathbb{Z}}
\renewcommand{\bar}{\overline}
\renewcommand{\div}{\mathrm{div}}
\renewcommand{\cD}{\mathcal{D}}

\renewcommand{\injlim}{\varinjlim}
\renewcommand{\P}{\mathbb{P}}
\renewcommand{\sp}{\mathrm{sp}}
\renewcommand{\tilde}{\widetilde}
\makeatletter
    
    \@addtoreset{equation}{section}
\makeatother
\newcommand{\inj}{\hookrightarrow}
\newcommand\lsub[2]{{\vphantom{#2}}_{#1}\!#2}
\newtheorem{thm}{Theorem}[section]
\newtheorem{prop}[thm]{Proposition}
\newtheorem{lem}[thm]{Lemma}

\newtheorem{cor}[thm]{Corollary}

\theoremstyle{definition}

\newtheorem*{notation}{Notation}
\newtheorem*{ack}{Acknowledgments}
\newtheorem*{funding}{Funding}

\theoremstyle{remark}
\newtheorem{rem}[thm]{Remark}

\begin{document}
\title{On the Brauer group of diagonal cubic surfaces}
\author{\textsc{Tetsuya Uematsu}}
\date{}
\maketitle
\begin{abstract}
We are concerned with finding explicit generators of the Brauer group of
 diagonal cubic surfaces in terms of norm residue symbols, which was
 originally studied by Manin.

We introduce the notion of uniform generators and find that the
 Brauer group of some classes of diagonal cubic surfaces have uniform generators. However, we also prove that the Brauer
 group of general diagonal cubic surfaces do not have such ones. This reveals that a result of Manin for certain diagonal cubic
 surfaces cannot be generalized in some sense.
\end{abstract}
\section{Introduction}\label{sec:Introduction}
\footnote[0]{Submitted: Oct. 5, 2012; Revised: Feb. 24, 2013}
\footnote[0]{%
2010 Mathematics Subject Classification. 
Primary: 14F22, 19F15. 
Secondary: 11R34, 14J26. 
}
Let $k$ be a field of characteristic zero and containing a fixed primitive
cubic root $\zeta$ of unity. In this paper, we study the
cohomological Brauer group of diagonal cubic surfaces $V$ over $k$, that
is, smooth projective surfaces defined by a homogeneous equation of the form
\[
 x^3+by^3+cz^3+dt^3=0,
\]
where $b,c,d \in k^{\ast}$. In particular, we are concerned with the
following two natural problems:
\begin{itemize}
\item[(1)] Determine the structure of $\Br(V)$ as an abelian group.
\item[(2)] Find generators of $\Br(V)$ in terms of norm
residue symbols.
\end{itemize}
In general, the Brauer group of a variety plays an important role in
studying its arithmetic and its geometry. For applications to the Hasse
principle, see for example, \cite{manin1971groupe} and
\cite{skorobogatov2001}. It is also used for studying zero-cycles
(\cite{lichtenbaum1969duality} and \cite{colliot1995arithmetique}). For
applications to the rationality problem, see \cite{artin1972some}. For such
studies, we want to know in advance the structure and generators of its
Brauer group.
 
For diagonal cubic surfaces, an original work in this direction was due
to Manin \cite{manin1986cubic}. He gave a complete answer to the above
two problems for diagonal cubic surfaces of the form
$x^3+y^3+z^3+dt^3=0$ for $d \in k^{\ast}\setminus (k^{\ast})^3$. Let
$\pi\colon V \to \Spec k$ be the structure morphism and put
$\Br(V)/\Br(k):=\Br(V)/\pi^{\ast}\Br(k)$. In this case, $\Br(V)/\Br(k)
\cong (\Z/3\Z)^2$ and
\begin{equation*}
\left\{d, \dfrac{x+\zeta y}{x+y}\right\}_3, 
\quad \left\{d, \dfrac{x+z}{x+y}\right\}_3
\end{equation*}
are its symbolic generators, where 
\begin{equation*}
 \{\cdot,\cdot\}_3\colon K_2^M(k(V)) \to
  H^2(k(V),\mu_3^{\otimes 2}) \cong H^2(k(V), \mu_3) \inj \Br(k(V)),
\end{equation*}
is a norm residue symbol map. As an application of these symbolic
generators, Saito and Sato \cite{saito2009zerocycle} recently computed
the degree-zero part of the Chow group of zero-cycles on such cubic
surfaces over $p$-adic fields explicitly, even in the case $p=3$.

In this paper, we study these problems in a more general setting where the
equation of $V$ is of the forms $x^3+y^3+cz^3+dt^3=0$ and
$x^3+by^3+cz^3+dt^3=0$. 

First, we prove the following theorem, which
gives an answer to the problems (1) and (2) for the case
$x^3+y^3+cz^3+dt^3=0$.
\begin{thm}\label{thm:1st}
 Let $k$ be as above and $V$ be the cubic
 surface over $k$ defined by an equation $x^3+y^3+cz^3+dt^3=0,$ where
 $c$ and $d \in k^{\ast}.$ Assume that $c,$ $d,$ $cd$
 and $d/c$ are not contained in $(k^{\ast})^3.$ Then we have the following:

{\rm(1)} The group $\Br(V)/\Br(k)$ is isomorphic to $\Z/3\Z.$

{\rm(2)} The element
\[
 e_1=\left\{\dfrac{d}{c}, \dfrac{x+\zeta y}{x+y}\right\}_3 \in \Br(k(V))
\]
is contained in $\Br(V).$

{\rm(3)} The image of $e_1$ in $\Br(V)/\Br(k)$ is a generator of this
	   group$.$
\end{thm}
The claim (1) is essentially due to
\cite{colliot1987arithmetique}. Recently Colliot-Th\'el\`ene and
Wittenberg found a symbolic generator of $\Br(V)/\Br(k)$ for
$V:x^3+y^3+2z^3=at^3$ when $k$ does not contain a primitive cubic root
of unity (\cite{colliot2012groupe}, Proposition 2.1). In this case, our symbolic
generator was also appeared in the proof of this proposition.

We note that in the result of Manin and Theorem \ref{thm:1st}, we can take a
generator {\it uniformly}. More precisely, let $c$ and $d$ be
indeterminates, $F=k(c,d)$, $V$ be the
cubic surface $x^3+y^3+cz^3+dt^3=0$ over $F$, and
\begin{equation*}
e(c,d)=\left\{\dfrac{d}{c},
 \dfrac{x+\zeta y}{x+y}\right\}_3
\end{equation*}
be an element in $\Br(V)$. Let $P=(c_0,d_0)$ be a point in
 $k^{\ast}\times k^{\ast}$ with $c_0$, $d_0$, $c_0d_0$ and $d_0/c_0$ not
 contained in $(k^{\ast})^3$, and $V_P$ the surface defined by
 $x^3+y^3+c_0z^3+d_0t^3=0$. If we want a symbolic generator of
 $\Br(V_P)/\Br(k)$, we can get it by specializing $e(c,d)$ at $P$. We
 denote this element by $\sp(e(c,d);P)$. A precise definition of the
 specialization will be given in \S \ref{sec:Brauer_group}. In general, it is not
 necessary that the Brauer group of a given variety has such uniform
 generators. 

Concerning the problem whether symbolic generators can be chosen
uniformly or not, we prove a non-existence result as stated below. Let
$F=k\left(b,c,d\right)$, where $b$, $c$, $d$ are indeterminates over
$k$, and let $V$ be the projective cubic surface over $F$ defined by the
equation $x^3+by^3+cz^3+dt^3=0$. For $P=(b_0,c_0,d_0)\in k^{\ast}\times
k^{\ast}\times k^{\ast}$, let $V_P$ be the projective cubic surface over
$k$ defined by the equation $x^3+b_0y^3+c_0z^3+d_0t^3=0$. For $e \in
\Br(V)$, we will define its specialization at $P$, $\sp(e;P) \in
\Br(V_P)$. Put
\begin{equation*}
\cP_k=\{P \in (\G_{m,k})^3 \mid \Br(V_P)/\Br(k) \cong \Z/3\Z.\}.
\end{equation*} 
Note (\cite{colliot1987arithmetique}) that $\Br(V_P)/\Br(k)$ isomorphic to
either of $0$, $\Z/3\Z$ and $(\Z/3\Z)^2$ and that Manin
dealt with the last case, as stated before.

We first prove the following:
\begin{thm}\label{thm:2nd}
Let $k, F$ and $V$ be as above. Then
\[
 \Br(V)/\Br(F)=0.
\]
\end{thm}
Note that this vanishingness does not follow directly from a seven-term
exact sequence induced by the Hochschild-Serre spectral sequence
\[
 E_2^{p,q}=H^p(F,H^q(\bar{V},\G_m)) \Rightarrow H^{p+q}(V,\G_m) 
\]
since $E_2^{1,1} \neq 0$, $V(F)=\emptyset$ and $\cd(F) \geq 3$. As far
as we know, this would be the first example of computation of Brauer
groups for such varieties.

As a corollary of Theorem \ref{thm:2nd}, we can obtain the following
non-existence result:
\begin{cor}\label{cor:2nd}
Let $k$, $F$ and $V$ as in Theorem \ref{thm:2nd}. Assume moreover
 $\dim_{\F_3} k^{\ast}/(k^{\ast})^3 \geq 2.$ Then there is no element $e
 \in \Br(V)$ satisfying the following condition$:$
 \begin{quote} 
  there exists a dense open subset $W \subset (\G_{m,k})^3$ such that
 $\sp(e;\cdot)$ is defined on $W(k) \cap \cP_k$ and for all $P \in W(k)
 \cap \cP_k,$ $\sp(e;P)$ is a generator of $\Br(V_P)/\Br(k).$
 \end{quote}
\end{cor}
We note that the assumption $\dim_{\F_3} k^{\ast}/(k^{\ast})^3 \geq 2$
is equivalent to the Zariski density of $\cP_k \subset (\G_{m,k})^3$, which is
essentially necessary to prove Theorem \ref{thm:2nd}. We easily see that
this assumption holds for various fields, for example, all finitely
generated fields over $\Q(\zeta)$ and $\Q_p(\zeta)$ for any prime number
$p$, and hence this is a mild assumption.

 This paper is written in the following fashion. In \S
\ref{sec:Brauer_group}, we describe the Brauer group of varieties in
terms of Galois cohomology. We also define specialization of Brauer
groups. In \S \ref{sec:diagonal_cubic_surface} we focus on diagonal
cubic surfaces, especially their Picard groups and Galois action on
them. In \S \ref{sec:11cd},
we give a proof of Theorem \ref{thm:1st}. Finally, in \S \ref{sec:1bcd}, we
prove Theorem \ref{thm:2nd} by computing the image under a differential
$d^{1,1}$ appearing in the above spectral sequence explicitly. We also 
discuss the condition $\dim_{\F_3} k^{\ast}/(k^{\ast})^3 \geq 2$
appearing in Corollary \ref{cor:2nd}. Finally, we prove Corollary \ref{cor:2nd}. 

\begin{notation}
For a group $A$ and $f \in \End(A)$, we denote by $\lsub{f}{A}$ the
kernel of $f$.

Throughout this paper, all fields are of characteristic zero. For a
field $k$, we denote a separable closure of $k$ by $\bar{k}$. We fix
such a field $\bar{k}$ and each algebraic separable extension of $k$ is
always considered as a subfield of this $\bar{k}$. If $k$ is a discrete
valuation field, the field $k^{\ur}$ denotes the maximal unramified
extension of $k$.

Fix a positive integer $n$ and assume that $k$ contains a primitive
$n$-th root $\zeta_n $ of unity. For $f$ and $g \in k^{\ast}$, we denote
by $\{f,g\}_n$ the image of $f \otimes g$ under a usual norm residue
symbol map
\begin{equation*}
 k^{\ast} \otimes k^{\ast} \to H^1(k,\mu_n) \otimes H^1(k, \mu_n)
  \stackrel{\cup}{\to} H^2(k,\mu_n^{\otimes 2}) \cong H^2(k,\mu_n) \cong \lsub{n}{\Br(k)},
\end{equation*}
where the third map is induced by $\zeta_n^i \otimes \zeta_n^j \mapsto
\zeta_n^{ij}$.

For a scheme $V$, all cohomology groups of $V$ mean \'etale cohomology.
\end{notation}
\section{Preliminaries of Brauer groups}\label{sec:Brauer_group}
Let $k$ be a field and $\pi\colon V \to \Spec k$ a variety over $k$. In
this section, we see two descriptions of $\Br(V)$ in terms of Galois
cohomology of $k$. We also introduce the notion of specialization of
Brauer groups.

First, we recall a fundamental exact sequence. By
the Hochschild-Serre spectral sequence
\begin{equation*}
 H^p(k,H^q(\bar{V},\G_m)) \Rightarrow H^{p+q}(V,\G_m),
\end{equation*}
we have the following exact sequence
\begin{equation}\label{eq:fund_exact_seq}
 0 \to \Br_1(V)/\Br(k) \to H^1(k,\Pic(\bar{V}))
 \stackrel{d^{1,1}}{\to} H^3(k,\bar{k}^{\ast}),
\end{equation}
where
\begin{equation*}
\Br_1(V):=\Ker(\Br(V) \to \Br(\bar{V})),\quad
\Br_1(V)/\Br(k):=\Br_1(V)/\pi^{\ast}\Br(k).
\end{equation*}
By this sequence, we know that $\Br_1(V)/\Br(k)$ has an inclusion into
$H^1(k,\Pic(\bar{V}))$. It is not clear whether this inclusion is an isomorphism or not. However, here are the following sufficient
conditions:
\begin{lem}\label{lem:isom_condition}
Let $V$ be a variety over a field $k$. If $\cd(k) \leq 2$ or $V(k) \neq \emptyset,$ then
\begin{equation*}
\Br_1(V)/\Br(k) \cong H^1(k,\Pic(\bar{V})).
\end{equation*}
\end{lem}
\begin{proof}
The first (resp.\ the second) assumption implies
 $H^3(k,\bar{k}^{\ast})=0$ (resp.\ $H^3(k,\bar{k}^{\ast}) \to
 H^3(V,\G_m)$
is injective), which implies the surjectivity of $\Br_1(V)/\Br(k) \to
 H^1(k,\Pic(\bar{V}))$.
\end{proof}

Secondly, we give another description of Brauer group of varieties. 
We use the following result in \S \ref{sec:11cd}. We describe the following
Brauer group
\begin{equation*}
\Br(V_L/V):=\Ker(\Br(V) \to \Br(V_L)),
\end{equation*}
where $L/k$ is a Galois extension. The claim is:
\begin{prop}\label{prop:another_desc_Brauer_group}
Let $V$ be a smooth$,$ geometrically integral variety over a field $k$ and let $L$ be a Galois extension of $k.$ Put
 $G:=\Gal(L/k).$ Then we have an exact sequence
\begin{equation*}
0 \to \Br(V_L/V) \to H^2(G,L(V)^{\ast}) \stackrel{\div}{\to} H^2(G,\Div(V_L)),
\end{equation*}
where $\div$ is naturally induced by
\begin{equation*}
 \div\colon L(V)^{\ast} \to \Div(V_L);\quad f \mapsto \div(f).
\end{equation*}
\end{prop}
\begin{proof}
Let $j\colon \eta=\Spec k(V) \to V$ be the generic point of $V$. We have the
 following exact sequence of \'etale sheaves on $V$:
\begin{equation}
0 \to \G_m \to j_{\ast}\G_{m,\eta} \to \Div_V \to 0,
\end{equation} 
where $\Div_V$ is the sheaf of Cartier divisors on $V$. By regularity of
 $V$, we have $H^1(V,\Div_V)=0$. Moreover, we have $H^i(V,j_{\ast}\G_{m,\eta})
 \cong H^i(k(V),\G_m)$ for all $i \geq 0$. These yield the commutative
 diagram with exact rows:
\begin{equation*}
\xymatrix{
0 \ar[r] & \Br(V) \ar[d] \ar[r] & \Br(k(V)) \ar[d] \ar[r] & H^2(V,\Div_V)
\ar[d] \\
0 \ar[r] & \Br(V_L) \ar[r] & \Br(L(V)) \ar[r] & H^2(V_L,\Div_V).
}
\end{equation*}
Taking the kernel of each column, we obtain the exact sequence
\begin{equation*}
0 \to \Br(V_L/V) \to \Br(L(V)/k(V)) \to \Ker(H^2(V,\Div_V)\to
 H^2(V_L,\Div_V)).
\end{equation*}
Applying the Hochschild-Serre spectral sequence 
\begin{equation*}
E^{p,q}_2=H^p(G,H^q(V_L,\cdot)) \Rightarrow E^{p+q}=H^{p+q}(V,\cdot)
\end{equation*}
to sheaves $j_{\ast}\G_{m,\eta} \to \Div_V$ on $V$, we have the following commutative diagram
\begin{equation*}
\xymatrix{
 H^2(G,L(V)^{\ast}) \ar[d] \ar[r]^(0.4){\cong} &\Ker(\Br(k(V))\to
 \Br(L(V))) \ar[d]\\
 H^2(G,\Div(V_L)) \ar[r]^(0.35){\cong} & \Ker(H^2(V,\Div_V)\to
 H^2(V_L,\Div_V)), 
}
\end{equation*}
which completes the proof of Proposition \ref{prop:another_desc_Brauer_group}.
\end{proof}

Finally, in the last of this section, we introduce the notion of
specialization of Brauer groups. In \S \ref{sec:11cd}, we will see that
the Brauer group of surfaces of the form $x^3+y^3+cz^3+dt^3=0$ has a
{\it uniform} symbolic generator, that is, if we put
\begin{equation*}
e(c,d)=\left\{\dfrac{d}{c},
 \dfrac{x+\zeta y}{x+y}\right\}_3
\end{equation*}
where $c$ and $d$ are considered as {\it indeterminates} and if we want
 a symbolic generator of $\Br(V_P)/\Br(k)$, where $V_P$ is the surface
 of the form $x^3+y^3+c_0z^3+d_0t^3=0$ with $c_0$ and $d_0 \in
 k^{\ast}$, we can get it by specializing $e(c,d)$ at
 $(c,d)=(c_0,d_0)$. 

To make this notion of uniformity precise, we define specialization as
follows. Let $k$ be a field, $\cO_F$ a polynomial ring over $k$ with $r$
variables, $F$ its fractional field and $f_1, \ldots f_m$ homogeneous
polynomials in $\cO_F[x_0,\ldots,x_n]$.  Let $\cV$ be the projective
scheme over $\cO_F$ defined as:
\begin{equation*}
\cV = \Proj\left(\cO_F[x_0,\ldots,x_n]/(f_1,\ldots,f_m)\right)
 \stackrel{\pi}{\to} \Spec \cO_F.
\end{equation*}    
Let $\pi_F\colon V:=\cV_F \to \Spec F$ be the base change of $\pi$ to
$\Spec F$. Assume that $V$ is smooth over $F$.  Let $e \in \Br(V)$ be an
arbitrary element. If $(S_i)_{i \in I}$ is the projective system of the
non-empty affine open subschemes in $\A_{k}^r=\Spec\cO_F$, we have
\begin{equation*}
 \projlim_i(\cV \times_{\A_k^r}S_i) \cong V,
\end{equation*}
and there exists a non-empty affine open subscheme $S$ and $\tilde{e}
\in \Br(\cV\times_{\A_k^r}S)$ satisfying that $\cV \times_{\A_k^r}S$ is
smooth over $S$ and that
\begin{equation*}
 \res^S_{\Spec F}(\tilde{e}) =e,
\end{equation*}
where $\res^S_{\Spec F}\colon \Br(\cV\times_{\A^r_k}S) \to \Br(V).$ This
  follows from \cite{grothendieck1967ega4_4}(Proposition 17.7.8) and
  \cite{milne1980etale}(III, Lemma 1.16).  For a given $P \in S(k)$, we
  have the following diagram:
\begin{equation*}
\xymatrix{
V_P \ar^{P}[r] \ar^{\pi_0}[d] \ar@{}|{\square}[dr] & \cV \times_{\A^r_k}S
 \ar^{\pi_S}[d] \ar@{}|{\square}[dr] & V \ar[l] \ar^{\pi_F}[d] \\
\Spec k \ar^{P}[r] & S & \Spec F \ar[l], \\
}
\end{equation*}
where $V_P:=\cV\times_{\A^r_k}\Spec k$.
We define the specialization of $e$ at $P$ as
\begin{equation*}
\sp(e;P):=P^{\ast}\tilde{e} \in \Br(V_P).
\end{equation*}
By the regularity of $\cV \times_{\A_k^r}S$, the map $\res^S_{\Spec F}$
is injective, which implies that this definition is independent of the
choice of $S$.
\section{Preliminaries of diagonal cubic surfaces}\label{sec:diagonal_cubic_surface}

In this section, we are concerned with diagonal cubic surfaces, in
particular, their Picard groups and their Galois structures. We mainly
use the same notation as in \cite{colliot1987arithmetique}. Let $k$ be a
field containing a primitive cubic root $\zeta$ of unity. Let $V$ be the
projective surface over $k$ defined by a homogeneous equation
\begin{equation*}
 ax^3+by^3+cz^3+dt^3=0,
\end{equation*}
where $a$, $b$, $c$ and $d$ are in $k^{\ast}$. Let $\pi\colon V\to \Spec k$
denote the structure morphism. Now we put
\begin{equation*}
 \lambda=\dfrac{b}{a},\quad \mu=\dfrac{c}{a}\quad \text{and} \quad \nu=\dfrac{ad}{bc},
\end{equation*}
and then we can write as the equation of $V$ 
\begin{equation*}
 x^3+\lambda y^3+\mu z^3+\lambda\mu\nu t^3=0.
\end{equation*}
We define some extensions of $k$ which are
frequently used in this paper. Let $\alpha$, $\alpha'$ and $\gamma$ be
solutions in $\bar{k}$ of equations $X^3-\lambda=0$, $X^3-\mu=0$ and
$X^3-\nu=0$ respectively. Put $\beta=\alpha\gamma$ and
$\beta'=\alpha'\gamma$. We define a field $k'$ and $k''$ as
$k(\alpha,\gamma)$ and $k'(\alpha')$.

$\bar{V}$ is a del Pezzo surface, obtained from $\bar{\P^2}$ by
blowing-up general 6 points. Thus $\Pic(\bar{V})$ is free of rank 7. 
We can write 27 lines on $\bar{V}$ explicitly:
\begin{align}\label{eq:27_lines}
 L(i)&\colon x+\zeta^i\alpha y=z+\zeta^i\beta t=0, \notag \\
 L'(i)&\colon x+\zeta^i\alpha y=z+\zeta^{i+1}\beta t=0, \notag\\
 L''(i)&\colon x+\zeta^i\alpha y=z+\zeta^{i+2}\beta t=0, \notag\\
 M(i)&\colon x+\zeta^i\alpha' z=y+\zeta^{i+1}\beta' t=0, \notag\\
 M'(i)&\colon x+\zeta^i\alpha' z=y+\zeta^{i+2}\beta' t=0, \\
 M''(i)&\colon x+\zeta^i\alpha' z=y+\zeta^i\beta' t=0, \notag\\
 N(i)&\colon x+\zeta^i\alpha\beta' t=y+\zeta^{i+2}\alpha^{-1}\alpha' z=0, \notag\\
 N'(i)&\colon x+\zeta^i\alpha\beta' t=y+\zeta^i\alpha^{-1}\alpha' z=0, \notag\\
 N''(i)&\colon x+\zeta^i\alpha\beta' t=y+\zeta^{i+1}\alpha^{-1}\alpha' z=0,\notag
\end{align}
where $i$ is either $0$, $1$ or $2$.

Since six lines $L(0)$, $L(1)$, $L(2)$, $M(0)$, $M(1)$ and $M(2)$ are
mutually skew, we can get a $\bar{k}$-morphism $\pi\colon \bar{V} \to
\bar{\P^2}$ by blowing down these six lines. We define $l \in
\Pic(\bar{V})$ as the inverse image of a line on $\bar{\P^2}$. Then we
can obtain generators of $\Pic(\bar{V})\cong \Z^7$:
\begin{equation}\label{eq:generator_of_Pic}
[L(0)],\quad [L(1)],\quad [L(2)],\quad [M(0)],\quad [M(1)],\quad
 [M(2)],\text{ and } l,
\end{equation} 
where $[D]$ denotes the class of $D \in \Div(\bar{V})$ in $\Pic(\bar{V})$. 
Let $H$ be the hyperplane section of $\bar{V}$ defined by the equation
$x=0$, 
\begin{equation*}
 [L]=[L(0)]+[L(1)]+[L(2)], \text{ and } [M]=[M(0)]+[M(1)]+[M(2)],
\end{equation*}
we have the following relation: 
\begin{equation}\label{eq:relation_of_HlLM}
[H]=3l - [L] - [M].
\end{equation} 
We have $\Pic(V_{k''}) \cong \Pic(\bar{V})^{G_{k''}} \cong \Z^7$. As its
generators, we take the classes corresponding to $[L(i)]$, $[M(i)]$ and
$l \in \Pic(\bar{V})$. By abuse of notation, we use the same symbols as
in the case of $\Pic(\bar{V})$.

By the geometrical rationality of $V$ and the birational invariance of
Brauer groups (\cite{grothendieck1968brauer3}, III, Th\'eor\`em 7.1), we have
$\Br(\bar{V}) \cong \Br(\bar{\P^2}) =0$. Hence we rewrite the
sequence~(\ref{eq:fund_exact_seq}) as follows:
\begin{equation}\label{eq:fund_exact_seq_rational}
0 \to \Br(V)/\Br(k) \to H^1(k,\Pic(\bar{V})) \stackrel{d^{1,1}}{\to}
 H^3(k,\bar{k}^{\ast})\quad \text{(exact)}.
\end{equation}
This sequence plays a fundamental role in this paper.  

The structure of the group
$H^1(k,\Pic(\bar{V}))$ is well-known:
\begin{prop}[\cite{colliot1987arithmetique}, Proposition 1.]\label{prop:ctks}
\begin{equation*}
 H^1(k,\Pic(\bar{V})) \cong \begin{cases} 0 & \text{if one of }\nu,
		  \nu/\lambda, \nu/\mu \text{ is a cube in }k^{\ast}, \\
		  (\Z/3\Z)^{2} & \text{if exactly three of }\lambda,
		  \mu, \lambda/\mu, \lambda\mu\nu, \lambda\nu, \mu\nu \\
		  &\text{are cubes in }k^{\ast}, \\ \Z/3\Z &
		  \text{otherwise}. \end{cases}
\end{equation*}
\end{prop}
\noindent
In \cite{colliot1987arithmetique}, this proposition is proved by showing
the following isomorphism
\[
 H^1(k'/k,\Pic(V_{k'})) \cong H^1(k,\Pic(\bar{V}))
\]
and reducing to explicit calculation of cohomology of the finite
extension $k'/k$ with coefficients $\Pic(V_{k'})$. Since $V(k')
\neq \emptyset$, we have $\Pic(V_{k'}) \cong \Pic(V_{k''})^{\Gal(k''/k')}$;
moreover using the explicit defining equations of divisors
(\ref{eq:27_lines}) and the equation (\ref{eq:relation_of_HlLM}), we see
that
\begin{equation}
 \Pic(V_{k'}) = \Z l \oplus \Z[L(0)] \oplus \Z[L(1)] \oplus \Z[L(2)]
  \oplus \Z[M].
\end{equation}
In the following, we give an explicit generating
cocycle of $H^1(k'/k,\Pic(V_{k'}))$ under some conditions stated below.

First, in \S\ref{sec:11cd}, we consider cubic surfaces over $k$ defined
by $x^3+y^3+cz^3+dt^3=0$. If one of $cd$ and $d/c$ is in $(k^{\ast})^3$,
we have $H^1(k,\Pic(\bar{V}))=0$ by Proposition \ref{prop:ctks} and
hence $\Br(V)/\Br(k)=0$. Therefore we need not consider this case. If
neither $cd$ nor $d/c$ is in $(k^{\ast})^3$ and $c$ is in
$(k^{\ast})^3$, such surfaces are isomorphic to surfaces defined by
$x^3+y^3+z^3+dt^3=0$ and Manin has already found the structure and the
generators of their Brauer groups. Hence we may assume $c$, $d$, $cd$
and $d/c \notin (k^{\ast})^3$.

Secondly, in \S\ref{sec:1bcd}, we consider surfaces $x^3+by^3+cz^3+dt^3=0$ over
$k(b,c,d)$, where $b$, $c$ and $d$ are indeterminates. Hence $k(\alpha,
\gamma,\alpha')/k$ is an extension of degree 27.

Therefore in the sequel of this section, we always assume one of the following
conditions:

(i) $\lambda=1$, and neither $\mu$, $\nu$, $\mu\nu$ nor $\nu/\mu$
	   is cubic in $k^{\ast}$.

(ii) $k(\alpha, \gamma,\alpha')$ is a field extension of $k$ with
	   degree 27.

By Proposition \ref{prop:ctks}, $H^1(k'/k,\Pic(V_{k'}))\cong \Z/3\Z$ in
both cases. Note that under the condition~(i), $k(\alpha)=k$ and hence
$k'/k$ is of degree $3$.

Let $s$ be the generator of $G=\Gal(k'/k(\alpha))$ such that
$s\gamma=\zeta\gamma$ and $w$ the generator of $\Gal(k''/k')$ such that
$w\alpha'=\zeta\alpha'$. Note that $G$ and $\Gal(k''/k')$ are isomorphic
to $\Z/3\Z$ under one of the above assumptions and such elements $s$ and
$w$ {\it do} exist. Moreover, under the condition (ii), let $t$ be the
generator of $\Gal(k'/k(\gamma))$ such that $t\alpha=\zeta\alpha$.

The claim is the following:
\begin{prop}\label{prop:explicit_cocycle}
{\rm (1)} As a generator of $H^1(G,\Pic(V_{k'}))$, we can take a class
 of the following cocycle $\phi':$
\[
\phi'(1) = 0,\quad 
\phi'(s) = [L(0)]-[L(2)],\quad 
\phi'(s^2) = [L(0)]-[L(1)].
\]

{\rm(2)} Under the condition {\rm(ii)}, as a generator of
 $H^1(k'/k,\Pic(V_{k'}))$, we can take a class of the following cocycle
 $\phi:$
\[
\phi((st)^i) = 0,\quad 
\phi(s(st)^i) = [L(0)]-[L(2)],\quad 
\phi(s^2(st)^i) = [L(0)]-[L(1)], 
\]
where $i$ takes on any values in $\{0,1,2\}$.
\end{prop}

\begin{proof}
(1) Since $G$ is a finite cyclic group, we have the following isomorphism:
\[
 H^1(G,\Pic(V_{k'})) \cong \hat{H}^{-1}(G,\Pic(V_{k'})) \cong 
\dfrac{\lsub{N_G}{\Pic(V_{k'})}}{I_G(\Pic(V_{k'}))}, 
\]
where the norm $N_G$ maps $x$ to $\sum_{s \in G} sx$ and the map $I_G$
maps $x$ to $sx-x$ for all $x \in \Pic(V_{k'})$. 
The action of $s$ on $\Pic(V_{k'})$ is as follows:
\[
 s\begin{pmatrix} l \\ [L(0)] \\ [L(1)] \\ [L(2)] \\ M \end{pmatrix}
=\begin{pmatrix} 4 & -1 & -1 & -1 & 2 \\
                 2 & -1 & -1 & 0 & -1 \\
                 2 & 0 & -1 & -1 & -1 \\
                 2 & -1 & 0 & -1 & -1 \\
                 3 & 0 & 0 & 0 & -2 
 \end{pmatrix}
 \begin{pmatrix} l \\ [L(0)] \\ [L(1)] \\ [L(2)] \\ M \end{pmatrix}.
\]

By a straightforward calculation, we know that
 $\lsub{N_G}{\Pic(V_{k'})}/I_G(\Pic(V_{k'})) \cong \Z/3\Z$ and
 $[L(1)]-[L(0)]$ is its generator. Computing the above isomorphisms
 explicitly, we easily find that $\phi'$ is a cocycle whose class in
 $H^1(G,\Pic(V_{k'}))$ is a generator.

Next we consider the claim (2). First, noting that $[L(1)]-[L(0)]$ is
 $st$-invariant, we can easily check $\phi$ is a cocycle.  Moreover, the
 image of $[\phi]$ under the restriction
\begin{equation*}
H^1(k'/k,\Pic(V_{k'})) \to H^1(G,\Pic(V_{k'})) \cong \Z/3\Z
\end{equation*}
is the class $[\phi']$ appearing in (1). Hence $\phi$ is also a non-zero
 element, in fact, a generator of $H^1(k'/k,\Pic(V_{k'}))$.
\end{proof}

In the last of this section, we introduce the following description of
$\Pic(V_{k'})$, which plays an important role in \S \ref{sec:11cd} and \S
\ref{sec:1bcd}. Let $\cD$ be the following free abelian group of rank 10:
\begin{equation*}
 \cD=\Z H \oplus \bigoplus_{i=0}^2 \Z L(i) \oplus \bigoplus_{i=0}^2\Z L'(i) \oplus \bigoplus_{i=0}^2\Z L''(i).
\end{equation*}
We see that $\cD$ is a $G$-submodule of $\Div(V_{k'})$ by using
(\ref{eq:27_lines}). Let $\cD_0$ be the $G$-submodule generated by the
following five divisors:
\begin{align*}
 D_1=\div(f_1)=(L(0)+L'(0)+L''(0))-H, \\
 D_2=\div(f_2)=(L(1)+L'(1)+L''(1))-H, \\
 D_3=\div(f_3)=(L(0)+L'(2)+L''(1))-H, \\
 D_4=\div(f_4)=(L(1)+L'(0)+L''(2))-H, \\
 D_5=\div(f_5)=(L(2)+L'(1)+L''(0))-H, 
\end{align*}
where
\begin{equation*}
 f_1=\dfrac{x+\alpha y}{x},\quad f_2=\dfrac{x+\zeta\alpha y}{x},\quad
 f_3=\dfrac{z+\beta t}{x},\quad f_4=\dfrac{z+\zeta\beta t}{x},\quad
 f_5=\dfrac{z+\zeta^2\beta t}{x}.
\end{equation*}
\begin{lem}\label{lem:exact_seq_of_Div_and_Pic}
Let $\cD$ and $\cD_0$ be as above$.$ Then we have the following exact
 sequence of $G$-modules$:$
\begin{equation*}
 0 \to \cD_0 \to \cD \to \Pic(V_{k'}) \to 0.
\end{equation*}
\end{lem}
\begin{proof}
 For the exactness at $\Pic(V_{k'})$, it suffices to show that we can
 write the classes $l$ and $[M]$ as linear combinations of $[H]$,
 $[L(i)]$, $[L'(i)]$ and $[L''(i)]$.  The intersection matrix with
 respect to the basis in (\ref{eq:generator_of_Pic}) is the diagonal
 matrix with entries $-1,-1,-1,-1,-1,-1$ and $1$. By using this matrix,
 (\ref{eq:27_lines}) and (\ref{eq:relation_of_HlLM}), we can write
 $[L'(0)]$ and $[L''(0)]$ as follows:
\begin{equation*}
 [L'(0)]=2l-[L(0)]-[L(1)]-[M],\quad [L''(0)]=l-[L(0)]-[L(2)],
\end{equation*}
which implies the surjectivity of $\cD \to \Pic(V_{k''})$.

The exactness at $\cD_0$ is trivial by definition and we also prove the
 exactness at $\cD$ by comparing the ranks of $\cD_0, \cD$ and
 $\Pic(V_{k'})$. This completes the proof of this lemma.
\end{proof}

\section{The case $x^3+y^3+cz^3+dt^3=0$}\label{sec:11cd}
In this section, let $k$ be a field containing a primitive cubic root $\zeta$ of
unity. The result in this section is:
\begin{thm}\label{thm:11cd}
 Let $V$ be the cubic surface over $k$ defined by a
homogeneous equation $x^3+y^3+cz^3+dt^3=0$, where $c$ and $d \in
k^{\ast}$. Moreover, we assume the condition {\rm(i)} in \S
{\rm\ref{sec:diagonal_cubic_surface}}, that is, $c$, $d$, $cd$ and $d/c$
are not in $(k^{\ast})^3$. Then we have the following$:$

 {\rm (1)} The group $\Br(V)/\Br(k)$ is isomorphic to $\Z/3\Z.$

 {\rm (2)} The symbol
\begin{equation*}
 e_1=\left\{\dfrac{d}{c}, \dfrac{x+\zeta y}{x+y}\right\}_3 \in \Br(k(V))
\end{equation*} 
is contained in $\Br(V).$

 {\rm (3)} The image of $e_1$ in $\Br(V)/\Br(k)$ is a generator of this
	   group$.$
\end{thm}
\begin{proof}
Let $G=\Gal(k'/k)$ and $s \in G$ a generator such that
$s\gamma=\zeta\gamma$. 

First we consider (1). This surface has a $k$-rational point
$P=(1:-1:0:0)$. Therefore the claim of (1) follows from Lemma
\ref{lem:isom_condition} and Proposition \ref{prop:ctks}.

Next we consider (2). Let $\phi'$ be as in Proposition
 \ref{prop:explicit_cocycle} (1). Computing
 the cocycle $\partial'\phi'$, where
\[
\partial'\colon H^1(G,\Pic(V_{k'})) \to H^2(G,k'(V)^{\ast}/k'^{\ast})
\]
is the connecting homomorphism induced by the exact sequence of $G$-modules:
\begin{equation}\label{eq:divisor}
 0 \to k'(V)^{\ast}/k'^{\ast} \to \Div(V_{k'}) \to \Pic(V_{k'}) \to 0,
\end{equation}
we can show
\[
 \partial'\phi'(s^i,s^j)=\left(\dfrac{f_2}{f_1}\right)^{a(i,j)} \in
 k'(V)^{\ast}/k'^{\ast}, \quad a(i,j):= \left\lfloor\dfrac{i+j}{3}\right\rfloor
		-\left\lfloor\dfrac{i}{3}\right\rfloor
		-\left\lfloor\dfrac{j}{3}\right\rfloor.
\]

On the other hand, the symbol $\{\nu, f_2/f_1\}_3 \in
\Br(k(V))$ is equal to 
\begin{equation*}
 (\chi_{3,\nu},f_2/f_1) \in H^2(k(V),\bar{k(V)}^{\ast}),
\end{equation*}
where $\chi_{3,\nu} \in \Hom(G_{k(V)},\Q/\Z) \cong H^2(k(V),\Z)$ is the
 cyclic character of order 3 associated to $\nu$ and $(\cdot,\cdot )$ be the symbol defined by
\[
(\cdot,\cdot)\colon H^2(k(V),\Z)
  \otimes H^0(k(V),\bar{k(V)}^{\ast}) \stackrel{\cup}{\to}
  H^2(k(V),\bar{k(V)}^{\ast});\quad (\chi,f) \mapsto f \cup \chi.
\]
For a proof, see \cite{serre1968corps}. Moreover, by the following commutative
diagram
\begin{equation*}
 \xymatrix{
  H^2(k(V),\Z) \otimes H^0(k(V),\bar{k(V)}^{\ast})
  \ar^(0.6){\cup}[r] & H^2(k(V),\bar{k(V)}^{\ast})  \\
  H^2(k'(V)/k(V),\Z) \otimes H^0(k'(V)/k(V),k'(V)^{\ast}) \ar[u] \ar^(0.62){\cup}[r] &
  H^2(k'(V)/k(V),k'(V)^{\ast}) \ar@{^(->}[u] \\
  H^2(G,\Z) \otimes H^0(G,k'(V)^{\ast}) \ar^{\cong}[u] \ar^(0.55){\cup}[r] &
   H^2(G,k'(V)^{\ast}), \ar^{\cong}[u] \\ 
 }
\end{equation*}
$(\chi_{\nu},f_2/f_1)$ can be considered as an element in
$H^2(G,k'(V)^{\ast})$, and we see that the corresponding cocycle is of the
form
\begin{equation*}
 \left\{\nu, \dfrac{f_2}{f_1}\right\}_3(s^i,s^j)=\left(\dfrac{f_2}{f_1}\right)^{a(i,j)} \in k'(V)^{\ast}.
\end{equation*}
Finally, we have the following commutative diagram with all rows and columns exact:
\begin{equation*}
\xymatrix{
  & \Br(k'/k) \ar@{=}[r] \ar[d] & \Br(k'/k) \ar[d] & \\
  0 \ar[r] & \Br(V_{k'}/V) \ar[r] \ar[d] &
   H^2(G,k'(V)^{\ast}) \ar^{\div}[r] \ar[d] &
  H^2(G,\Div(V_{k'})) \ar@{=}[d] \\
  0 \ar[r] & H^1(G,\Pic(V_{k'})) \ar^{\partial'}[r] &
  H^2(G,k'(V)^{\ast}/k'^{\ast}) \ar^{\div}[r] & H^2(G, \Div(V_{k'})), \\
}
\end{equation*}
where the middle column is induced by the exact sequence
\begin{equation*}
0 \to k'^{\ast} \to k'(V)^{\ast} \to k'(V)^{\ast}/k'^{\ast} \to 0,
\end{equation*}
the middle row is the result of Proposition
\ref{prop:another_desc_Brauer_group}, the bottom row is induced by the
exact sequence (\ref{eq:divisor}) and the triviality of its leftmost
term follows from the fact that the action of $G$ on $\Div(V_{k'})$ maps
one basis to another. Then the map
\begin{equation*} 
\Br(V_{k'}/V) \to H^1(G,\Pic(V_{k'}))
\end{equation*}
is naturally induced by  
\begin{equation*}
H^2(G,k'(V)^{\ast}) \to H^2(G,k'(V)^{\ast}/k'^{\ast}).
\end{equation*}
The fact that $\partial'\phi'$ and $\left\{\nu,
 f_2/f_1\right\}_3$ coincide in $H^2(G,k'(V)^{\ast}/k'^{\ast})$
 shows that
\begin{equation*}
 \left\{\nu,f_2/f_1\right\}_3=\left\{\dfrac{d}{c},\dfrac{x+\zeta y}{x+y}\right\}_3 \in
 \Br(V_{k'}/V) \subset \Br(V),
\end{equation*}
which completes the proof of (2).

Finally we consider (3). By the above argument, $[\phi']$ and 
$\left\{\nu, f_2/f_1\right\}_3$ 
coincide in $H^1(G,\Pic(V_{k'}))$. Hence we can take 
\begin{equation*}
 \left\{\dfrac{d}{c}, \dfrac{x+\zeta y}{x+y}\right\}_3 
\end{equation*}
as a generator of the group $\Br(V)/\Br(k)$. This completes the proof of
 Theorem \ref{thm:11cd}.
\end{proof}

By using the specialization of Brauer groups, we can formulate Theorem
\ref{thm:11cd} as follows:
\begin{cor}
 Let $k$ be as in Theorem {\rm \ref{thm:11cd}}$,$ $\cO_F=k[c,d],$ $F$ its
 fractional field and 
\begin{equation*}
 V=\Proj(F[x,y,z,t]/(x^3+y^3+cz^3+dt^3)).
\end{equation*}
 Then
\begin{equation*}
 e_1=\left\{\dfrac{d}{c},\dfrac{x+\zeta y}{x+y}\right\}_3 \in \Br(V)
\end{equation*}   
is a uniform generator$,$ that is$,$ for all $P=(c_0,d_0) \in
 k^{\ast}\times k^{\ast}$ such that
 $c_0,$ $d_0,$ $c_0d_0$ and $d_0/c_0 \notin (k^{\ast})^3,$
$\sp(e_1;P)$ is a generator of $\Br(V_P)/\Br(k).$
\end{cor}
\begin{proof}
We confirm that $\sp(e_1;P)$ is in fact the desired symbol, that is, 
\begin{equation*}
 \left\{\dfrac{d_0}{c_0},\dfrac{x+\zeta y}{x+y}\right\}_3.
\end{equation*}
 We define $S \subset \A_k^2$ and $\cV$ to be:
\begin{align*}
 S&=\G_{m,k} \times \G_{m,k}=\Spec k[c^{\pm},d^{\pm}],\\
 \cV&=\Proj \cO_F[x,y,z,t]/(x^3+y^3+cz^3+dt^3).
\end{align*}
where the symbol $c^{\pm}$ is the abbreviation of $c$ and $c^{-1}$. Put
 $\cV \times S:=\cV \times_{\A_k^2} S$. we see that $\cV \times S$ is smooth over
 $S$. Moreover, we note that $e_1$ can lift to $\tilde{e} \in
 \Br(\cV\times S)$.  This follows from concrete calculations using the
 following exact sequence, which is a consequence of the absolute purity
 due to Gabber~\cite{fujiwara2002proof}:
\begin{equation*}
 0 \to \Br(\cV\times S) \to \Br(F(V)) \stackrel{\bigoplus_{x}
  \res_x}{\longrightarrow} \bigoplus H^1(\kappa(x),\Q/\Z),
\end{equation*}
where the sum is taken over all points $x$ of codimension one in
 $\cV\times S$, and
 $\kappa(x)$ is the residue field of $x$.
Hence we can use the above $S$ and $\tilde{e}$ to construct the
 specialization of $e_1$.

Now we define the subscheme $U$ of $\cV \times S$ as follows:
\begin{equation*}
  U=\cV\times S \setminus (D_{+}(x+y) \cup D_{+}(x+\zeta y)),
\end{equation*}
where $D_{+}(f)$ is the non-vanishing locus of a homogeneous polynomial
 $f$. Explicitly,
 if we put
\begin{equation*}
R:=\dfrac{k[c^{\pm},d^{\pm}]\left[\dfrac{x}{x+y}, \dfrac{y}{x+y},
		      \dfrac{z}{x+y}, \dfrac{t}{x+y},
		      \dfrac{x+y}{x+\zeta
		      y}\right]}{\left(\dfrac{x^3+y^3+cz^3+dt^3}{(x+y)^3}\right)},
\end{equation*}
then $U=\Spec R$. We have
\begin{equation*}
\dfrac{d}{c},\quad \dfrac{x+\zeta y}{x+y} \in \Gamma(U,\cO_U)^{\ast}
\end{equation*}
and hence
\begin{equation*}
 \left\{\dfrac{d}{c}, \dfrac{x+\zeta y}{x+y}\right\}_3 \in H^2(U,\mu_3), 
\end{equation*}
where
\begin{equation*}
 \{\cdot,\cdot\}_3 \colon \Gamma(U,\cO_U)^{\ast} \otimes \Gamma(U,\cO_U)^{\ast} \to
  H^1(U,\mu_3) \otimes H^1(U,\mu_3) \stackrel{\cup}{\to}
  H^2(U,\mu_3^{\otimes 2})
  \cong H^2(U,\mu_3)
\end{equation*}
is norm residue map defined similarly as in the field case. 
Take any $P=(c_0,d_0) \in k^{\ast} \times k^{\ast}$ such that $c_0, d_0,
 c_0d_0$ and $c_0/d_0 \notin (k^{\ast})^3$ and put $R_P:=R/(c-c_0,d-d_0)$.
We have the canonical morphism $P^{\ast}\colon U \to U_P:=\Spec R_P$ and the
 following commutative diagram:
\begin{equation*}
\xymatrix{
\Br(U_P) & \Br(U) \ar[l]_{P^{\ast}} \ar[r] & \Br(F(V)) \\
\Br(V_P) \ar@{^(->}[u]_{\res^{V_P}_{U_P}} & \Br(\cV \times S)
 \ar[l]_{P^{\ast}} \ar@{^(->}[u]_{\res_U^{\cV\times S}}
 \ar[r] & \Br(V) \ar@{^(->}[u]_{\res_{F(V)}^V} \\
}
\end{equation*}
Therefore we get
\begin{align*}
 \res^{V_P}_{U_P}(\sp(e_1;P))
&=\res^{V_P}_{U_P}(P^{\ast}(\tilde{e}))\\
&=P^{\ast}(\res^{\cV\times S}_U(\tilde{e}))\\
&=P^{\ast}\left(\left\{\dfrac{d}{c},\dfrac{x+\zeta
 y}{x+y}\right\}_3\right)\\
&=\left\{\dfrac{d_0}{c_0},\dfrac{x+\zeta y}{x+y}\right\}_3
\end{align*}
and complete the proof.
\end{proof}
\section{The case $x^3+by^3+cz^3+dt^3=0$}\label{sec:1bcd}

In this section, let $k$ be a field containing a primitive cubic root
$\zeta$ of unity, $\lambda$, $\mu$ and $\nu$ indeterminates,
$\cO_F=k[\lambda,\mu,\nu]$, $F=k(\lambda,\mu,\nu)$ and
\begin{equation*}
 \cV=\Proj(\cO_F[x,y,z,t]/(x^3+\lambda y^3+\mu z^3+\lambda\mu\nu t^3).
\end{equation*} 
Put $V=\cV \times_{\A_k^3} \Spec F$. For all $P \in (\G_{m,k})^3(k)$,
$V_P=\cV \times_{\A_k^3}\Spec k(P)$ is smooth over $k$. In this section,
we are mainly devoted to proving the following:
\begin{thm}\label{thm:vanishing_of_Brauer_group}
\begin{equation*} 
\Br(V)/\Br(F)=0.
\end{equation*} 
\end{thm}

As a corollary of this result, we obtain the non-existence of uniform
generators. We define the set $\cP_k$ to be
\begin{equation*}
\{P \in (\G_{m,k})^3(k) \mid \Br(V_P)/\Br(k) \cong \Z/3\Z\}.
\end{equation*}
Here we have the following
\begin{prop}\label{prop:C(k)}
The following conditions are equivalent$:$

{\rm(1)} $\cP_k$ is Zariski dense in $(\G_{m,k})^3;$

{\rm(2)} $\cP_k$ is non-empty$;$

{\rm(3)} $\dim_{\F_3}k^{\ast}/(k^{\ast})^3 \geq 2.$
\end{prop}
We define $C(k)$ to be the above equivalent conditions. We
easily see that the condition (3) holds for various fields, for example,
\begin{itemize}
\item any finitely generated field over $\Q(\zeta)$ or $\Q_p(\zeta)$ for
      any prime number $p$;
\item a function field of any variety over $\mathbb{C}$ or
	   $\mathbb{R}$ of dimension $\geq 1$.
\end{itemize}
Thus this condition $C(k)$ is mild and reasonable.

Now we can state the following non-existence result:
\begin{cor}\label{cor:nonexistance_of_generator}
Let $k$ and $V$ be as above. Assume moreover
$\dim_{F_3}k^{\ast}/(k^{\ast})^3 \geq 2$. Then there is no element $e
\in \Br(V)$ satisfying the following property$:$
\begin{quote}
 there exists a dense open subset $W \subset (\G_{m,k})^3$ such that
 $\sp(e;\cdot)$ is defined on $W(k) \cap \cP_k$ and for all $P \in W(k)
 \cap \cP_k, \sp(e;P)$ is a generator of $\Br(V_P)/\Br(k).$
\end{quote}
\end{cor}  
\begin{rem}
 Let $V$ be a cubic surface over $k$ of the form
 $x^3+by^3+cz^3+dt^3=0$. Assume $H^1(k,\Pic(\bar{V}))=\Z/3\Z$ and
 $V(k)=\emptyset$. We note some known results of the structure and
 generators of $\Br(V)/\Br(k)$.
 \begin{enumerate}
  \item By a theorem of Merkurjev-Suslin \cite{merkurjev1982}, we always
	write a generator of $\Br(V)/\Br(k)$ as a sum of norm residue
	symbols.
  \item If $\cd(k) \leq 2$, $\Br(V)/\Br(k)$ is isomorphic to
	$\Z/3\Z$. Its symbolic generator is not ``uniform'' by Corollary
	\ref{cor:nonexistance_of_generator} and we do not know this can
	be written by \textit{one} symbol $\{f,g\}_3$ for some $f$, $g
	\in k(V)^{\ast}$.  However, here is a partial result. Let $k$ be
	a field satisfying the following condition:
	\begin{quote}
	 For any cubic extension $L$ of $k$, the restriction $\Br(k) \to
	 \Br(L)$ is surjective.
	\end{quote}
	Some examples of $k$ are 
	\begin{itemize}
	 \item a field $k$ with $\cd(k) \leq 1$.
	 \item a local field $k$.
	\end{itemize}
	Then we find that a generator can be
       taken as $\{d/bc,f\}_3$ for some $f \in k(V)^{\ast}$. 
  \item If $\cd(k) \geq 3$, it is difficult to determine whether
	$\Br(V)/\Br(k)$ is isomorphic to $0$ or $\Z/3\Z$. Our $V/F$ has
	$\cd(F)\geq 3$ and $V(F) =\emptyset$ and
	$H^1(F,\Pic(\bar{V}))\neq 0$. As far as we know, Our result
	would be the first example of computation of the Brauer
	group of such varieties.
 \end{enumerate}
\end{rem}  
\begin{proof}[Proof of Theorem {\rm\ref{thm:vanishing_of_Brauer_group}}.]
We recall some notations in \S \ref{sec:diagonal_cubic_surface}. We
define
\begin{equation*}
\alpha=\sqrt[3]{\lambda}, \quad \gamma=\sqrt[3]{\nu}, \quad
 \alpha'=\sqrt[3]{\mu},\quad \beta=\alpha\gamma.
\end{equation*}
Moreover we put
\begin{equation*}
 F'=F(\alpha,\gamma),\quad F''=F'(\alpha')=F(\alpha,\gamma,\alpha').
\end{equation*}
We have the following exact sequence:
\begin{equation*}
 0 \to \Br(V)/\Br(F) \to H^1(F,\Pic(\bar{V})) \stackrel{d^{1.1}}{\to} H^3(F,\bar{F}^{\ast}).
\end{equation*}
and we know $H^1(F,\Pic(\bar{V}))\cong H^1(F'/F,\Pic(V_{F'}))\cong
 \Z/3\Z$. Therefore, to prove the theorem, it suffices to show the image
 of $\phi \in H^1(F'/F,\Pic(V_{F'}))$ in Proposition
 \ref{prop:explicit_cocycle} (2) does not vanish in
 $H^3(F,\bar{F}^{\ast})$.

Before proving this claim, we sketch an outline of its proof. The proof
 consists of 4 steps. In Step 1, we compute the image of $\phi$ under
 the differential
\begin{equation*}
 d^{1,1}\colon H^1(F'/F,\Pic(V_{F'})) \to H^3(F'/F,(F')^{\ast})
\end{equation*} 
explicitly. Since the inflation $i^{F'}_{\bar{F}}:H^3(F'/F,F'^{\ast})
 \to H^3(F,\bar{F}^{\ast})$ does not necessarily injective, if we prove
 that $d^{1,1}(\phi)\neq 0$ in $H^3(F'/F,F'^{\ast})$, this is
 insufficient to prove the theorem. In Step 2, we consider
 $d^{1,1}(\phi)$ as an element of $H^3(F''/F,\mu_3)$. In Step 3, by
 computing the residue of $d^{1,1}(\phi)$ along a certain prime divisor
 $D$ in $\A_k^3$ and replacing the base field $k$ with the field adding
 all roots of unity, we reduce the proof to showing that a certain
 cocycle induced by $\phi$ is nontrivial in $H^2(k(D),\mu_3)$, where
 $k(D)$ is the function field of $D$. Finally, in Step 4, we again
 compute the residue of the cocycle in Step 3 along a certain prime
 divisor $D'$ in $D$ and check this is nonzero. These steps complete the
 proof of the theorem.

{\bf Step 1.} Let  
\begin{align*}
&\partial\colon H^1(F'/F,\Pic(V_{F'})) \to H^2(F'/F,\cD_0), \\
&\delta\colon H^2(F'/F,\cD_0) \to H^3(F'/F, F'^{\ast})
\end{align*}
be connecting homomorphisms induced by the exact sequence in Lemma
\ref{lem:exact_seq_of_Div_and_Pic} and 
\[
 0 \to F'^{\ast} \to
\div^{-1}(\cD_0) \to \cD_0 \to 0.
\] 
We have
\[
 d^{1,1}=\delta \circ \partial \colon H^1(F'/F,\Pic(V_{F'})) \to H^3(F'/F,F'^{\ast})
\]
by \cite{kresch2008effectivity}, Proposition 6.1. First we compute the
 cocycle $\partial\phi \in Z^2(F'/F,\cD_0)$. Let $\cD$ and $\cD_0$ be as
 in \S \ref{sec:diagonal_cubic_surface}. We take
\begin{equation*}
0,\quad \ L(0)-L(2),\quad L(0)-L(1) \in \cD
\end{equation*}
as lifts of $0$, $[L(0)]-[L(2)]$ and $[L(0)]-[L(1)] \in \Pic(V_{F'})$ respectively, and note
 that 
\begin{equation*}
\div\dfrac{x+\zeta^2\alpha
 y}{x}=\div\left(-\dfrac{\mu f_3f_4f_5}{f_1f_2}\right) \in \cD_0.
\end{equation*}　
From the construction of the map $\partial$, we get the following equations
\begin{alignat*}{3}
 &\partial\phi(1,1)=0, & \ 
 &\partial\phi(1,s)=0,&\  
 &\partial\phi(1,s^2)=0,\\
 &\partial\phi(s,1)=0, & \ 
 &\partial\phi(s,s)=\div\dfrac{z+\zeta\beta t}{x+\zeta^2\alpha y},& 
 &\partial\phi(s,s^2)=\div\dfrac{x+\alpha y}{z+\zeta^2\beta t},\\
 &\partial\phi(s^2,1)=0, &
 &\partial\phi(s^2,s)=\div\dfrac{x+\alpha y}{z+\zeta\beta t},&
 &\partial\phi(s^2,s^2)=\div\dfrac{z+\zeta^2\beta t}{x+\zeta\alpha y},\\
 &\partial\phi(t,1)=0, &
 &\partial\phi(t,s)=0,&
 &\partial\phi(t,s^2)=0,\\
 &\partial\phi(st,1)=0, &
 &\partial\phi(st,s)=\div\dfrac{z+\zeta^2\beta t}{x+\alpha y},&
 &\partial\phi(st,s^2)=\div\dfrac{x+\zeta\alpha y}{z+\beta t},\\
 &\partial\phi(s^2t,1)=0, &
 &\partial\phi(s^2t,s)=\div\dfrac{x+\zeta\alpha y}{z+\zeta^2\beta t},&
 &\partial\phi(s^2t,s^2)=\div\dfrac{z+\beta t}{x+\zeta^2\alpha y},\\
 &\partial\phi(t^2,1)=0, & 
 &\partial\phi(t^2,s)=0,&
 &\partial\phi(t^2,s^2)=0,\\
 &\partial\phi(st^2,1)=0, &
 &\partial\phi(st^2,s)=\div\dfrac{z+\beta t}{x+\zeta\alpha y},&
 &\partial\phi(st^2,s^2)=\div\dfrac{x+\zeta^2\alpha y}{z+\zeta\beta t},\\
 &\partial\phi(s^2t^2,1)=0, &
 &\partial\phi(s^2t^2,s)=\div\dfrac{x+\zeta^2\alpha y}{z+\beta t},&
 &\partial\phi(s^2t^2,s^2)=\div\dfrac{z+\zeta\beta t}{x+\alpha y},
\end{alignat*}
\begin{equation*}
\partial\phi(s^{i_1}t^{j_1},s^{i_2}t^{j_2})=
  \partial\phi(s^{i_1}t^{j_1},s^{i_2-j_2}),
\end{equation*}
where the indices $i_1$, $i_2$, $j_1$ and $j_2$ take on any values in
$\{0, 1, 2\}$.

Sending this cocycle under $\delta$, we get $\delta\partial\phi$ in
 $Z^3(F'/F,F'^{\ast})$. If we take 
\begin{equation*}
1,\quad \dfrac{x+\zeta^i\alpha y}{z+\zeta^j\beta t}, \quad
 \dfrac{z+\zeta^j\beta t}{x+\zeta^i\alpha y} \in \div^{-1}(\cD_0) 
\end{equation*}
as lifts of $0$, $\div\dfrac{x+\zeta^i\alpha y}{z+\zeta^j\beta t}$ and
$\div\dfrac{z+\zeta^j\beta t}{x+\zeta^i\alpha y} \in \cD_0$ respectively,
this cocycle is determined by the following equations: 
\begin{alignat*}{3}
 &\delta\partial\phi(t^{j_1},s^{i_2}t^{j_2},s^{i_3}t^{j_3})=1,\\
 &\delta\partial\phi(s^{i_1}t^{j_1},1,s^{i_3}t^{j_3})=1,\\
 &\delta\partial\phi(s^{i_1}t^{j_1},s^{i_2}t^{j_2},1)=1,\\
 &\delta\partial\phi(st^{j_1},s,s)=1, & \ 
 &\delta\partial\phi(st^{j_1},s,s^2)=-\mu, \\ 
 &\delta\partial\phi(st^{j_1},s^2,s)=1, & \ 
 &\delta\partial\phi(st^{j_1},s^2,s^2)=-\mu^{-1},\\
 &\delta\partial\phi(st^{j_1},t,s)=1, & \ 
 &\delta\partial\phi(st^{j_1},t,s^2)=-\mu^{-1}, \\ 
 &\delta\partial\phi(st^{j_1},st,s)=-\mu^{-1}, & \ 
 &\delta\partial\phi(st^{j_1},st,s^2)=-\mu,\\
 &\delta\partial\phi(st^{j_1},s^2t,s)=-\mu, & \ 
 &\delta\partial\phi(st^{j_1},s^2t,s^2)=1, \\ 
 &\delta\partial\phi(st^{j_1},t^2,s)=-\mu, & \ 
 &\delta\partial\phi(st^{j_1},t^2,s^2)=1,\\
 &\delta\partial\phi(st^{j_1},st^2,s)=-\mu^{-1}, & \ 
 &\delta\partial\phi(st^{j_1},st^2,s^2)=1, \\ 
 &\delta\partial\phi(st^{j_1},s^2t^2,s)=1, & \ 
 &\delta\partial\phi(st^{j_1},s^2t^2,s^2)=1,\\
 &\delta\partial\phi(s^2t^{j_1},s,s)=-\mu^{-1}, & \ 
 &\delta\partial\phi(s^2t^{j_1},s,s^2)=1, \\ 
 &\delta\partial\phi(s^2t^{j_1},s^2,s)=-\mu, & \ 
 &\delta\partial\phi(s^2t^{j_1},s^2,s^2)=1,\\
 &\delta\partial\phi(s^2t^{j_1},t,s)=1, & \ 
 &\delta\partial\phi(s^2t^{j_1},t,s^2)=-\mu, \\ 
 &\delta\partial\phi(s^2t^{j_1},st,s)=1, & \ 
 &\delta\partial\phi(s^2t^{j_1},st,s^2)=1,\\
 &\delta\partial\phi(s^2t^{j_1},s^2t,s)=1, & \ 
 &\delta\partial\phi(s^2t^{j_1},s^2t,s^2)=-\mu^{-1}, \\ 
 &\delta\partial\phi(s^2t^{j_1},t^2,s)=-\mu^{-1}, & \ 
 &\delta\partial\phi(s^2t^{j_1},t^2,s^2)=1,\\
 &\delta\partial\phi(s^2t^{j_1},st^2,s)=1, & \ 
 &\delta\partial\phi(s^2t^{j_1},st^2,s^2)=-\mu, \\ 
 &\delta\partial\phi(s^2t^{j_1},s^2t^2,s)=-\mu, & \ 
 &\delta\partial\phi(s^2t^{j_1},s^2t^2,s^2)=-\mu^{-1},
\end{alignat*}
\begin{equation*}
\delta\partial\phi(s^{i_1}t^{j_1},s^{i_2}t^{j_2},s^{i_3}t^{j_3})=\delta\partial\phi(s^{i_1}t^{j_1},s^{i_2}t^{j_2},s^{i_3-j_3}),
\end{equation*}
where the indices $i_1$, $i_2$, $i_3$, $j_1$, $j_2$ and  $j_3$ take on
 any values in $\{0, 1, 2\}$.

{\bf Step 2.} Let $i^{F'}_{\bar{F}}$ be the inflation
\begin{equation*}
i^{F'}_{\bar{F}}\colon H^3(F'/F,F'^{\ast}) \to H^3(F,\bar{F}^{\ast}).
\end{equation*}
The class $i^{F'}_{\bar{F}}\delta\partial[\phi$] in
 $H^3(F,\bar{F}^{\ast})$ is a $3$-torsion element, hence
 by the Kummer sequence, $i^{F'}_{\bar{F}}\delta\partial[\phi]$ comes from
 $H^3(F,\mu_3)$.
Now we find a finite extension $K$ over $F$ such that
 $i^{F'}_{\bar{F}}\delta\partial[\phi]$ comes from $H^3(K/F,\mu_3)$. In
 fact, we can take $K=F''$:
\begin{prop}\label{prop:h3mu3}
 The class $i^{F'}_{\bar{F}}\delta\partial[\phi] \in
 H^3(\bar{F}/F,\bar{F}^{\ast})$ comes from $H^3(F''/F,\mu_3).$
\end{prop}
\begin{proof}
 The exact sequence of $\Gal(F''/F)$-modules
\begin{equation*}
 1 \to \mu_3 \to F''^{\ast} \stackrel{3}{\to} (F''^{\ast})^3 \to 1,
\end{equation*}
yields the following commutative diagram
\begin{equation*}
\xymatrix{
 & H^3(F'/F,F'^{\ast}) \ar^{i^{F'}_{F''}}[d] &  \\
 H^3(F''/F,\mu_3) \ar[d] \ar[r] & H^3(F''/F,F''^{\ast})
  \ar^{i^{F''}_{\bar{F}}}[d] \ar^{3}[r] & H^3(F''/F,(F''^{\ast})^3) \ar[d]\\
 H^3(F,\mu_3) \ar[r] & H^3(F,\bar{F}^{\ast}) \ar^{3}[r] &
  H^3(F,\bar{F}^{\ast}), \\
}
\end{equation*}
where $i_{F'}^{F''}$ and $i^{F''}_{\bar{F}}$ are inflations and each row
 is exact. To prove the claim, it suffices to
 show $i^{F'}_{F''}\delta\partial[\phi]$ vanishes in
 $H^3(F''/F,(F''^{\ast})^3)$. Let $w$ be the generator of $\Gal(F''/F')$
 defined as \S \ref{sec:diagonal_cubic_surface}. The image of
 $i^{F'}_{F''}\delta\partial[\phi]$ under $3\colon H^3(F''/F,F''^{\ast})
 \to H^3(F''/F,(F''^{\ast})^3)$ is the class of the following cocycle:
\begin{equation*}
 (s^{i_1}t^{j_1}w^{k_1},s^{i_2}t^{j_2}w^{k_2},s^{i_3}t^{j_3}w^{k_3})
  \mapsto \delta\partial\phi(s^{i_1}t^{j_1},s^{i_2}t^{j_2},s^{i_3}t^{j_3})^3,
\end{equation*}
and what we have to prove is that this cocycle is in $B^3(F''/F,(F''^{\ast})^3)$.
Define $\psi \in C^2(F''/F,(F''^{\ast})^3)$ to be:
\begin{alignat*}{2}
&\psi(t^{j_1}w^{k_1},s^{i_2}t^{j_2}w^{k_2})=1, &\quad 
&\psi(s^{i_1}t^{j_1}w^{k_1},w^{k_2})=1, \\
&\psi(st^{j_1}w^{k_1},sw^{k_2})=-\mu^{-1}, &\quad 
&\psi(st^{j_1}w^{k_1},s^2w^{k_2})=-\mu, \\
&\psi(s^2t^{j_1}w^{k_1},sw^{k_2})=-\mu, &\quad 
&\psi(s^2t^{j_1}w^{k_1},s^2w^{k_2})=-\mu^{-1}, 
\end{alignat*}
\begin{equation*}
\psi(s^{i_1}t^{j_1}w^{k_1},s^{i_2}t^{j_2}w^{k_2})
=\psi(s^{i_1}t^{j_1}w^{k_1},s^{i_2-j_2}w^{k_2}),
\end{equation*}
where indices $i_{\ast}$, $j_{\ast}$ and $k_{\ast}$ take on any value in $\{0, 1, 2\}$.
Then we can easily see $d\psi = (i^{F'}_{F''}\delta\partial\phi)^3$ in
 $C^3(F''/F,(F''^{\ast})^3)$ and hence the class of $i^{F'}_{F''}\delta\partial\phi$
 vanishes in $H^3(F''/F,(F''^{\ast})^3)$. This completes the proof of
 Proposition \ref{prop:h3mu3}. 
\end{proof}

By using this cochain $\psi$, we can construct a cocycle $\Phi \in
Z^3(F''/F,\mu_3)$ whose image in $H^3(F,\bar{F}^{\ast})$ is
$i^{F''}_{\bar{F}}\delta\partial[\phi]$ in a usual manner. 
As a lift $\tilde{\psi}$ of $\psi$, we can take the following cochain:
\begin{alignat*}{2}
&\tilde{\psi}(t^{j_1}w^{k_1},s^{i_2}t^{j_2}w^{k_2})=1, &\quad 
&\tilde{\psi}(s^{i_1}t^{j_1}w^{k_1},w^{k_2})=1, \\
&\tilde{\psi}(st^{j_1}w^{k_1},sw^{k_2})=-\alpha'^{-1}, &\quad 
&\tilde{\psi}(st^{j_1}w^{k_1},s^2w^{k_2})=-\alpha', \\
&\tilde{\psi}(s^2t^{j_1}w^{k_1},sw^{k_2})=-\alpha', &\quad 
&\tilde{\psi}(s^2t^{j_1}w^{k_1},s^2w^{k_2})=-\alpha'^{-1}, 
\end{alignat*}
\begin{equation*}
\tilde{\psi}(s^{i_1}t^{j_1}w^{k_1},s^{i_2}t^{j_2}w^{k_2})
=\tilde{\psi}(s^{i_1}t^{j_1}w^{k_1},s^{i_2-j_2}w^{k_2}).
\end{equation*}
Then we can take the cocycle $\Phi \in Z^3(F''/F,\mu_3)$ explicitly as follows:
\begin{equation*}
 (s^{i_1}t^{j_1}w^{k_1},s^{i_2}t^{j_2}w^{k_2},s^{i_3}t^{j_3}w^{k_3})
  \mapsto
  \dfrac{\tilde{\psi}(s^{i_2}t^{j_2}w^{k_2},s^{i_3}t^{j_3}w^{k_3})}{w^{k_1}\tilde{\psi}(s^{i_2}t^{j_2}w^{k_2},s^{i_3}t^{j_3}w^{k_3})}
  \in \mu_3.
\end{equation*} 

{\bf Step 3.} For any prime divisor $D \subset \A_k^3=\Spec
 k[\lambda,\mu,\nu]$, we have the following commutative diagram:
\begin{equation*}
\xymatrix{
H^3(F''/F,\mu_3) \ar^{i^{F''}_{\bar{F}}}[d]& \\
H^3(F,\mu_3) \ar[d] \ar^(0.45){\res_D}[r] & H^2(k(D),\Z/3\Z) \ar[d]\\
H^3(F,\Q/\Z(1)) \ar_{\cong}[d] \ar^(0.45){\res_D}[r] & H^2(k(D),\Q/\Z) \\
H^3(F,\bar{F}^{\ast}), &
}
\end{equation*}
where $F=k(\lambda,\mu,\nu)$ is considered as the function field of
 $\A_k^3$, $k(D)$ is the function field of $D$, and $\res_D$ are residue
 maps associated to $D$.

Recall that our goal is to prove the nontriviality of 
 $i^{F'}_{\bar{F}}\delta\partial[\phi] \in
 H^3(F,\bar{F}^{\ast})$. To prove this, by the above diagram, it suffices to show:
\begin{equation}\label{eq:residue_is_nonzero_in_H2}
\text{There exists } D \subset \A^3_{k} \text{ such that }
 \res_D(i^{F''}_{\bar{F}}[\Phi]) \neq 0 \in H^2(k(D),\Q/\Z).
\end{equation} 

In the sequel, $D$ always denotes the divisor $\{\mu=0\} \subset
 \A^3_k$. Let $\cO_D$ be the completion of the local ring
 $k[\lambda,\mu,\nu]_{(\mu)}$ at its maximal ideal and $F_D$ its
 fractional field. Note that $\mu$ is a uniformizer of $\cO_D$ and the
 residue field of $\cO_D$ is isomorphic to $k(D)=k(\lambda,\nu)$.
 
 Now we should recall the definition of $\res_D$. There is the canonical isomorphism
\begin{equation*}
\iota\colon \Hom(G_{F_D^{\ur}},\mu_3) =H^1(F_D^{\ur},\mu_3) \cong
 {F_D^{\ur}}^{\ast}/({F_D^{\ur}}^{\ast})^3 \cong \Z/3\Z,
\end{equation*}
where the middle isomorphism is induced by Kummer sequence and the right one
 is given by normalized valuation on $F_D^{\ur}$. Then $\res_D$ is given by
\begin{equation*}
H^3(F,\mu_3)\to H^3(F_D,\mu_3) \stackrel{r}{\to}
 H^2(k(D),\Hom(G_{F_D^{\ur}},\mu_3)) \stackrel{H^2(\iota)}{\to} H^2(k(D),\Z/3\Z).
\end{equation*}
For an explicit description of the residue map $r$, see
\cite{garibaldi2003cohomological}(III, Theorem 6.1).

Now we describe the class $r[i^{F''}_{\bar{F}}\Phi] \in
H^2(k(D),\Hom(G_{F^{\ur}_D},\mu_3))$ explicitly. By the definition of
$r$ and the fact $i^{F''}_{\bar{F}}\Phi$ originally comes from the
cocycle $\Phi$ of $\Gal(F''/F)$, we would naturally expect that
$ri^{F''}_{\bar{F}}\Phi$ also comes from the cocycle of the Galois group
of residue fields of $F''/F$ along to $D$. In fact, we find that it is
true.

 Before stating the claim, we introduce some field extensions. Let
 $k(D)'$, $F_D''$, $F_D'$ be the same notation as in \S
 \ref{sec:diagonal_cubic_surface}. Moreover, by abuse of notation, we
 denote the elements in $\Gal(F_D''/F_D)$ corresponding to $s$, $t$ and
 $w \in \Gal(F''/F)$ as the same symbols.  To make our situation clear,
 we give the following diagram of field extensions:
\begin{equation*}
\xymatrix{
& & F_D'' \ar@{.>}^(0.45){{\rm residue\ field}}[rr]
\ar@{-}^{\rm ramified}_{3}[d] & & k(D)' \ar@{=}[d] \\
F'' \ar[urr] \ar@{-}_{3}[d] & & F_D' \ar@{.>}[rr]
 \ar@{-}^{\rm unramified}_{9}[d] & &
 k(D)' \ar@{-}_{9}[d] \\
F' \ar[urr] \ar@{-}_{9}[d] & & F_D \ar@{.>}[rr] & & k(D) \\
F \ar_(0.4){{\rm completion}}[urr] & & & &\\
}
\end{equation*}

The claim is:
\begin{lem}\label{lem:rPhi}
If we define the cochain
\begin{equation*}
\bar{r\Phi} \in C^2(k(D)'/k(D),\Hom(\Gal(F_D''/F_D'),\mu_3))
\end{equation*}
as
\begin{equation*}
\bar{r\Phi}(\bar{s}^{i_1}\bar{t}^{j_1},\bar{s}^{i_2}\bar{t}^{j_2})(w^k):=\Phi(w^k,s^{i_1}t^{j_1},s^{i_2}t^{j_2}),
\end{equation*}
where $\bar{s}$ and $\bar{t}$ is the image of $s$ and $t$ under the
 natural map
\begin{equation*}
  \Gal(F_D''/F_D) \to \Gal(k(D)'/k(D)),
\end{equation*}
then $\bar{r\Phi}$ is a cocycle and its image under the map
\begin{equation*}
  i^{k(D)'}_{\bar{k(D)}}\colon H^2(k(D)'/k(D),\Hom(\Gal(F_D''/F_D'),\mu_3)) \to
   H^2(k(D),\Hom(G_{F_D^{\ur}},\mu_3))
\end{equation*}
is $r[i^{F''}_{\bar{F}}\Phi].$
\end{lem}
\begin{proof}
we can prove that $\bar{r\Phi}$ is a cocycle by a straightforward
 calculation. The latter claim is easy to check by the definition of $r$
 and the proof is left to the reader.
\end{proof}
By using natural isomorphisms $\iota$ and 
\[
 \Hom(\Gal(F_D''/F_D'),\mu_3) \cong \Z/3\Z; \quad (w \mapsto \zeta)
 \mapsto 1,
\]
we rewrite $i^{k(D)'}_{\bar{k(D)}}$ simply as follows:
\[
 H^2(k(D)'/k(D),\Z/3\Z) \to H^2(k(D),\Z/3\Z).
\]

For a field $K$ of characteristic $0$, we denote $\tilde{K}$ by
 $\bigcup_{n>0} K(\zeta_n)$, where $\zeta_n$ is a primitive $n$-th root
 of unity. Noting that $k(D)'=k(D)(\alpha,\gamma)$ and that $\alpha$ and
 $\gamma$ are transcendental over $k$, we have $k(D)' \cap \tilde{k(D)}
 = k(D)$ and therefore
\begin{equation*}
 \Gal(\tilde{k(D)'}/\tilde{k(D)}) \stackrel{\cong}{\to} \Gal(k(D)'/k(D)).
\end{equation*}
We fix an isomorphism $\Q/\Z \cong \Q/\Z(1)$ as trivial
 $\tilde{k(D)}$-modules. Then we have the following
 commutative diagram:
\begin{equation*}
\xymatrix{
H^2(k(D)'/k(D),\Z/3\Z) \ar_{\cong}[d] \ar[r] & H^2(k(D),\Q/\Z) \ar[d] \\
H^2(\tilde{k(D)'}/\tilde{k(D)},\Z/3\Z) \ar_{\cong}[d] \ar[r] &
 H^2(\tilde{k(D)},\Q/\Z) \ar^{\cong}[d] \\
H^2(\tilde{k(D)'}/\tilde{k(D)},\mu_3) \ar[d] \ar[r] &
H^2(\tilde{k(D)},\Q/\Z(1)) \ar^{\cong}[d] \\
H^2(\tilde{k(D)},\mu_3) \ar[r] & H^2(\tilde{k(D)},\bar{k(D)}^{\ast}).
}
\end{equation*}
Since the bottom map in the above diagram is injective by Hilbert's
 Theorem 90, in order to prove the claim
 (\ref{eq:residue_is_nonzero_in_H2}), it suffices to show:
\begin{equation*}
[\bar{r\Phi}] \in H^2(k(D)'/k(D),\Z/3\Z) \text{ does not vanish in
 }H^2(\tilde{k(D)},\mu_3).
\end{equation*}

{\bf Step 4.} For simplicity, we put
 $E=\tilde{k(D)}=\tilde{k}(\lambda,\nu)$ and
 $E'=\tilde{k(D)'}=E(\alpha,\gamma)$. we define the cocycle $\Psi \in
 Z^2(E'/E,\mu_3)$ as follows:
\begin{align*}
\Psi(t^{j_1},s^{i_2}t^{j_2})=1, \quad
\Psi(st^{j_1},s^{i_2}t^{j_2})=
\begin{cases} 1 & j_2=0\\ \zeta^2 & j_2=1\\ \zeta & j_2=2,\end{cases}\quad
\Psi(s^2t^{j_1},s^{i_2}t^{j_2})=
\begin{cases} 1 & j_2=0\\ \zeta & j_2=1\\ \zeta^2 & j_2=2.\end{cases}
\end{align*}
We can easily check that $[\Psi]$ is the image of $[\bar{r\Phi}] \in
 H^2(k(D)'/k(D),\Z/3\Z)$ under the isomorphism $H^2(k(D)'/k(D),\Z/3\Z)
 \stackrel{\cong}{\to} H^2(E'/E,\mu_3)$ in the above diagram.

What we have to show is that the image of $[\Psi] \in H^2(E'/E,\mu_3)$ under
\begin{equation*}
i^{E'}_{\bar{E}}\colon H^2(E'/E,\mu_3) \to H^2(E,\mu_3)
\end{equation*}
is nonzero. This is a consequence of the following:
\begin{prop} \label{prop:residue_is_nonzero_in_H1}
Put $D'=\{\nu=0\} \subset \A^2_k$. The image of $i^{E'}_{\bar{E}}[\Psi]$ under the residue map
\begin{equation*}
\res_{D'}\colon H^2(E,\mu_3) \to H^1(k(D'),\Z/3\Z)
\end{equation*}
is nonzero.
\end{prop}
\begin{proof}
We fix notations. Let $\cO_{D'}$ be the completion of the local ring
 $k[\lambda,\nu]_{(\nu)}$ at its maximal ideal and $E_{D'}$ its
 fractional field. Note that $\nu$ is a uniformizer of $\cO_{D'}$ and
 the residue field of $\cO_{D'}$ is isomorphic to
 $k(D')=k(\lambda)$. Let $E_{D'}'$ be the same notation as in \S
 \ref{sec:diagonal_cubic_surface}. By abuse of notation, we denote the
 elements in $\Gal(E_{D'}'/E_{D'})$ corresponding to $s$ and $t \in
 \Gal(E'/E)$ as the same symbols.  To make our situation clear, we give
 the following diagram of field extensions:
\begin{equation*}
\xymatrix{
& & E_{D'}' \ar@{.>}^(0.45){{\rm residue\ field}}[rr]
\ar@{-}^{\rm ramified}_{3}[d] & & k(D')(\alpha) \ar@{=}[d] \\
E' \ar[urr] \ar@{-}_{3}[d] & & E_{D'}(\alpha) \ar@{.>}[rr]
 \ar@{-}^{\rm unramified}_{3}[d] & &
 k(D')(\alpha) \ar@{-}_{3}[d] \\
E(\alpha) \ar[urr] \ar@{-}_{3}[d] & & E_{D'} \ar@{.>}[rr] & & k(D'). \\
E \ar_(0.4){{\rm completion}}[urr] & & & &\\
}
\end{equation*}
Now $\res_{D'}$ is given by
\begin{equation*}
H^2(E,\mu_3)\to H^2(E_{D'},\mu_3) \stackrel{r}{\to}
 H^1(k(D'),\Hom(G_{E_{D'}^{\ur}},\mu_3)) \stackrel{\cong}{\to} H^1(k(D'),\Z/3\Z).
\end{equation*}

We also have a similar result to Lemma \ref{lem:rPhi}:
\begin{lem}\label{lem:rPsi}
If we define the cochain
\begin{equation*}
\bar{r\Psi} \in C^1(k(D')(\alpha)/k(D'),\Hom(\Gal(E_{D'}'/E_{D'}(\alpha)),\mu_3))
\end{equation*}
as
\begin{equation*}
\bar{r\Psi}(\bar{t}^{j})(s^i):=\Psi(s^i,t^j),
\end{equation*}
where $\bar{t}$ is the image of $t$ under the
 natural map
\begin{equation*}
  \Gal(E_{D'}'/E_{D'}) \to \Gal(k(D')(\alpha)/k(D')),
\end{equation*}
then $\bar{r\Psi}$ is a cocycle and its image under the map
\begin{equation*}
  i^{k(D')(\alpha)}_{\bar{k(D')}}\colon H^1(k(D')(\alpha)/k(D'),\Hom(\Gal(E_{D'}'/E_{D'}(\alpha)),\mu_3)) \to
   H^1(k(D'),\Hom(G_{E_{D'}^{\ur}},\mu_3))
\end{equation*}
is $ri^{E'}_{\bar{E}}[\Psi].$
\end{lem}
\begin{proof}
 The claim follows from similar calculations in Lemma
 \ref{lem:rPhi}. The details are left to the reader.
\end{proof}
We now go back to the proof of Proposition
\ref{prop:residue_is_nonzero_in_H1}. We know that
$i^{k(D')(\alpha)}_{\bar{k(D')}}$ is injective. Moreover, we can easily
check $[\bar{r\Psi}] \neq 0$ by definition. Therefore
$\res_{D'}(i^{E'}_{\bar{E}}[\Psi]) \neq 0$, which completes the proof of
Proposition \ref{prop:residue_is_nonzero_in_H1}.
\end{proof}

Theorem \ref{thm:vanishing_of_Brauer_group} is a consequence of
Proposition \ref{prop:residue_is_nonzero_in_H1}.
\end{proof}
Next we give a proof of Proposition \ref{prop:C(k)}. Before proving
the proposition, we note the following.
\begin{lem}\label{lem:density}
 Let $S_0, S_1$ and $S_2$ be infinite subsets of
 $k^{\ast}.$ Then $S_0\times S_1\times S_2$ is Zariski dense in
 $(\G_{m,k})^3.$
\end{lem}
\begin{proof}
It suffices to show for any sufficient small open subscheme $W$ in
 $(\G_{m,k})^3$, the intersection $(S_0 \times S_1 \times S_2) \cap W$
 is non-empty. Since all open subschemes $U_0 \times U_1 \times U_2
 \subset (\G_{m,k})^3$ form an open base, where $U_i$ run through affine
 open subschemes in $\G_{m,k}$, we may assume that $W$ is of the form $U
 \times U \times U$ for an affine open subscheme $U \subset
 \G_{m,k}$. Hence it suffices to prove that $S_0 \cap U \neq \emptyset$
 for any sufficiently small affine open $U \subset \G_{m,k}$. We may take
 $U$ as an affine open subscheme of the form:
\begin{equation*}
U=\Spec k[\lambda^{\pm}]_f,\quad 0 \neq f \in
 k[\lambda].
\end{equation*}
Since $S_0$ is infinite, there exists $\lambda_0 \in S_0$ such that
 $f(\lambda_0) \neq 0$. Then $\lambda_0 \in S_0 \cap U$ and in
 particular $S_0 \cap U$ is non-empty.
\end{proof}

\begin{proof}[Proof of Proposition {\rm\ref{prop:C(k)}}]
(1) $\Rightarrow$ (2). This is a trivial implication.

(2) $\Rightarrow$ (3). We prove the contrapositive statement. If we assume
 $\dim_{\F_3}k^{\ast}/(k^{\ast})^3=1$, we can take $v \in
 k^{\ast}\setminus (k^{\ast})^3$. Then the equation of diagonal cubic
 surfaces is essentially equal to one of the following:
\begin{alignat*}{3}
 &x^3+y^3+z^3+t^3=0, \quad
 &&x^3+y^3+z^3+vt^3=0,\quad
 &&x^3+y^3+z^3+v^2t^3=0,\\
 &x^3+y^3+vz^3+vt^3=0,\quad
 &&x^3+y^3+vz^3+v^2t^3=0,
\end{alignat*}
all of which have a $k$-rational point. We also see by
Proposition \ref{prop:ctks} that $H^1(k,\Pic(\bar{V_P}))\cong 0$ or $(\Z/3\Z)^2$ for all
 $P \in (\G_{m,k})^3$, and therefore
\begin{equation*}
\forall P \in (\G_{m,k})^3, \ \Br(V_P)/\Br(k) \cong 0 \mbox{ or } (\Z/3\Z)^2   
\end{equation*}
by Lemma \ref{lem:isom_condition}. We can also prove the case
 $\dim_{\F_3}k^{\ast}/(k^{\ast})^3=0$ in a similar way. Hence we have
 $\cP_k=\emptyset$.

(3) $\Rightarrow$ (1). We first construct a subset $\cP$ of $\A_k^3$
 satisfying the following three conditions:
\begin{itemize}
 \item[(i)] $\cP$ is Zariski dense in $\A_k^3$;
 \item[(ii)] $P \in \cP \Rightarrow V_P(k) \neq \emptyset$;
 \item[(iii)] $P \in \cP \Rightarrow H^1(k,\Pic(\bar{V_P})) \cong
	     \Z/3\Z$.
\end{itemize}
Since $\dim_{\F_3}k^{\ast}/(k^{\ast})^3\geq 2$, we can take two linearly
 independent elements $v_1$ and $v_2$. Now we define $\cP$ as
\begin{equation*}    
 \cP=S_0\times S_1\times S_2, \quad S_0=(k^{\ast})^3,\quad
  S_1=v_1(k^{\ast})^3,\quad S_2=v_2(k^{\ast})^3.
\end{equation*}
We show that $\cP$ satisfies the above three conditions. First, by Lemma
 \ref{lem:density} and the assumption that $(k^{\ast})^3$ is infinite,
 the condition (i) holds. Secondly, for $P=(\lambda_0, \mu_0, \nu_0) \in
 \cP$, we can take $\lambda_0' \in k^{\ast}$ such that
 $(\lambda_0')^3=\lambda_0$, and $V_P$ has a $k$-rational point
 $(\lambda_0':-1:0:0)$. Hence the condition (ii) holds. Finally, by the
 choice of $v_1$ and $v_2 \in k^{\ast}$ and Proposition \ref{prop:ctks},
 we can see that the condition (iii) holds.

Conditions (ii), (iii) and Lemma \ref{lem:isom_condition} imply
 $\Br(V_P)/\Br(k) \cong \Z/3\Z$ for all $P \in \cP$ and therefore we
 complete the proof of (1) $\Rightarrow$ (3).

This completes the proof of Proposition \ref{prop:C(k)}.  
\end{proof}

\begin{proof}[Proof of Corollary {\rm\ref{cor:nonexistance_of_generator}}] 
We would have an element $e \in \Br(V)$ satisfying the property stated
 in Corollary \ref{cor:nonexistance_of_generator}:
\begin{quote}
 there exists a dense open subset $W \subset (\G_{m,k})^3$ such that
 $\sp(e;\cdot)$ is defined on $W(k) \cap \cP_k$ and for all $P \in
 W(k)\cap \cP_k$, $\sp(e;P)$ is a generator of $\Br(V_P)/\Br(k)$.
\end{quote}
By Theorem \ref{thm:vanishing_of_Brauer_group}, we have
\begin{equation*}
 \Br(V)/\Br(F)=0
\end{equation*}
and hence there exists an element $e' \in \Br(F)$ such that
 $\pi_F^{\ast}e'=e$. We have the isomorphism
 \begin{equation*}
  \injlim_{i} \Br(S_i) = \Br(F),
\end{equation*}
where $(S_i)$ is the projective system of the non-empty open affine
 subschemes in $\A^3_{k}$, and there exists a non-empty affine open
 subscheme $S$ and $\tilde{e'} \in \Br(S)$ such that $\tilde{e'}$ is a
 lift of $e'$ and $\cV\times_{\A_k^3}S$ is smooth over $S$. Since
 $\cP_k$ is a Zariski dense set in $(\G_{m,k})^3$ by Proposition
 \ref{prop:C(k)} and $S \cap W$ is a non-empty Zariski open set in
 $(\G_{m,k})^3$, there exists a point $P \in (S \cap W)(k) \cap
 \cP_k$. For this point $P$, we have the following commutative diagram:
\begin{equation*}
\xymatrix{
\Br(V_P) &
\ar^{P^{\ast}}[l] \Br(\cV\times_{\A^3_k}S) \ar@{^(->}[r] & 
\Br(V) \\
\Br(k) \ar_{\pi_P^{\ast}}[u] &
\ar^{P^{\ast}}[l] \Br(S) \ar_{\pi_{S}^{\ast}}[u] \ar@{^(->}[r]&
\Br(F) \ar_{\pi_{F}^{\ast}}[u], \\
}
\end{equation*}
and hence we can take $\pi_S^{\ast}\tilde{e'}$ as a lift of $e$. Then we
 get
\begin{equation*}
\sp(e;P) = P^{\ast}(\pi_S^{\ast}\tilde{e'}) =
 \pi_P^{\ast}P^{\ast}\tilde{e'} \in \pi_P^{\ast}\Br(k).
\end{equation*}
This means that $\sp(e;P)$ is zero in the group $\Br(V_P)/\Br(k)$, which
 contradicts that $\sp(e;P)$ is a generator of $\Br(V_P)/\Br(k) \cong
 \Z/3\Z$. Therefore we see that there is no such element $e$, and
 complete the proof of Corollary \ref{cor:nonexistance_of_generator}.
\end{proof}

\begin{funding}
This work was supported by the Global COE Program ``The research and training center
 for new development in mathematics'' at Graduate School of Mathematical
 Sciences, The University of Tokyo.
\end{funding}
\begin{ack} 
The author is grateful to Professor Shuji Saito for his encouragement,
for reading a draft version of this paper and for refining the statement
of the main theorem. He thanks Professor Jean-Louis Colliot-Th\'el\`ene
for answering his questions, for giving valuable ideas, and for pointing out
some mistakes. He is also indebted to the referee for carefully reading
and for giving valuable comments.

The final publication of this article is available at Oxford University
 Press (Quarterly Journal of Mathematics), via
\url{http://qjmath.oxfordjournals.org/content/65/2/677}.
\end{ack}

{\scriptsize
\textsc{Tetsuya Uematsu}\\
\textsc{Graduate School of Mathematical Sciences, The University of Tokyo}\\
\textsc{3-8-1 Komaba Meguro-ku Tokyo 153-8914, JAPAN}\\
\textit{e-mail address}: tetsuya1@ms.u-tokyo.ac.jp\\
Current Address:\\
\textsc{Department of General Education, Toyota National College of
Technology}\\
\textsc{2-1 Eisei-cho Toyota Aichi 471-8525, JAPAN}\\
\textit{e-mail address}: utetsuya@08.alumni.u-tokyo.ac.jp}
\end{document}